\documentclass[11pt,a4paper]{article}
\usepackage{amsfonts,amsmath,amssymb,amsthm,amsxtra,color,graphicx,latexsym,mathtools,xcolor,enumerate}
\usepackage{geometry}

\newtheorem{definition}{Definition}
\newtheorem{proposition}[definition]{Proposition}
\newtheorem{remark}[definition]{Remark}
\newtheorem{corollary}[definition]{Corollary}

\numberwithin{equation}{section}
\numberwithin{definition}{section}

\newcommand{\Iff}{if\textcompwordmark f}

\newcommand{\rn}{\ensuremath{\left(\varPhi,\vect{y}\right)}}
\newcommand{\rns}{\ensuremath{\left(\varPhi^*\!,\vect{y}\right)}}

\newcommand{\vect}[1]{\boldsymbol{#1}}
\newcommand{\TU}[1]{\textup{#1}}
\DeclareMathOperator{\ID}{id}
\DeclareMathOperator{\RANK}{rank}
\DeclareMathOperator{\DIFF}{d\!}

\DeclareMathOperator{\TR}{tr}

\begin{document}
\title{On the shape operator of relatively parallel hypersurfaces in the $n$-dimensional relative differential geometry}

\author{{Stylianos Stamatakis and Ioannis Kaffas}\\ \emph{Aristotle University of Thessaloniki}\\ \emph{Department of Mathematics}\\ \emph{GR-54124 Thessaloniki, Greece}\\  \emph{e-mail: stamata@math.auth.gr}}
\date{}
\maketitle

\begin{abstract}

\noindent We deal with hypersurfaces in the framework of the $n$-dimensional relative differential geometry.

\noindent We consider a hypersurface $\varPhi$ of $\mathbb{R}^{n+1}$ with position vector field $\vect{x}$, which is relatively normalized by a relative normalization $\vect{y}$. 
Then $\vect{y}$ is also a relative normalization of every member of the one-parameter family $\mathcal{F}$ of hypersurfaces $\varPhi_\mu$ with position vector field $$\vect{x}_\mu = \vect{x} + \mu \, \vect{y},$$ where $\mu$ is a real constant.
We call every hypersurface $\varPhi_\mu \in \mathcal{F}$ relatively parallel to $\varPhi$ at the ``relative distance" $\mu$.
In this paper we study \begin{enumerate}[(a)]
                         \item the shape (or Weingarten) operator,
                         \item the relative principal curvatures,
                         \item the relative mean curvature functions and
                         \item the affine normalization
                       \end{enumerate}
of a relatively parallel hypersurface $\left( \varPhi_\mu,\vect{y}\right)$ to $\left(\varPhi,\vect{y}\right)$. 
\medskip

\noindent\emph{Key Words}: relative and equiaffine differential geometry, hypersurfaces in the Euclidean space, Blaschke hypersurfaces in affine differential geometry, Peterson correspondence, relative mean curvature functions

\medskip
\noindent\emph{MSC 2010}: 53A05, 53A15, 53A40
\end{abstract}

\section{Preliminaries}\label{Section 1}
            To set the stage for this work we present briefly in this section  the main definitions, formulae and results on relative differential geometry; for this purpose we have used \cite{Schirokow} and \cite{Simon1991} as  general references.

            We consider a $C^{r}$-hypersurface $\varPhi = (M,\vect{x})$ in $\mathbb{R}^{n+1} $
            defined by an $n$-dimensional, oriented, connected $C^{r}$-manifold $M$, $r \geq 3$, and by a $C^{r}$-immersion  $\vect{x} \colon M \rightarrow \mathbb{R}^{n+1}$,   whose Gaussian curvature $\widetilde{K}$ never vanishes on $M$. 
            
            Let $\vect{\xi}$ be the unit normal vector field to $\varPhi$ and
\begin{equation}
            II \coloneqq -  \langle \DIFF \vect{x}, \DIFF \vect{\xi} \rangle \eqqcolon h_{ij}\DIFF u^i \DIFF u^j, \quad i,j = 1,\dotsc,n,                     \label{1.10}
\end{equation}
            be 
            the second 
            fundamental form of $\varPhi$,
            where $\langle \, , \, \rangle$ denotes the standard scalar product in $\mathbb{R}^{n+1}$ and $(u^1,u^2, \dotsc, u^n) \in M$ are local coordinates.

            We denote by $\partial_i f$, $\partial_j\partial_i f$ etc. the partial derivatives of a $C^r$-function (or a vector-valued function) $f$ with respect to $u^i$.
            
            A  $C^{r}$-mapping $\vect{y} \colon M \rightarrow \mathbb{R}^{n+1}$ is called a $C^{r}$-relative normalization of $\varPhi$, if
\begin{subequations}\label{1.15}
\begin{align}
            \RANK \left(\Big\{\partial_1 \vect{x}, \dotsc, \partial_n \vect{x}, \vect{y}\Big\}\right) &= n + 1,\\
            \RANK \left( \Big \{ \partial_1 \vect{x}, \dotsc, \partial_n \vect{x}, \partial_i \vect{y} \Big\} \right) &= n \quad \forall \,\, i = 1, \dotsc, n,
\end{align}
\end{subequations}
            The pair $(\varPhi,\vect{y})$ is called a relatively normalized hypersurface in $\mathbb{R} ^{n+1}$ and the straight line issuing from a point $P \in \varPhi$ in the direction of $\vect{y}$ is called the relative normal of \rn{} at $P$.

            The covector $\vect{X}$ of the tangent vector space is defined by
\begin{equation}\label{1.20}
            \langle \vect{X},\partial_i\vect{x}\rangle = 0 \quad
             \text{and}\quad \langle \vect{X},\vect{y}\rangle = 1.
\end{equation}

            The quadratic  form
\begin{equation}\label{1.25}
            G \coloneqq - \langle \DIFF \vect{x}, \DIFF \vect{X}  \rangle
\end{equation}
            is called the relative metric of \rn{}. 
            
            For its coefficients $G_{ij}$ the following relations hold
\begin{equation}\label{1.30}
            G_{ij} = - \langle  \partial_i \vect{x}, \partial_j \vect{X} \rangle = \langle  \partial_j \partial_i \vect{x}, \vect{X} \rangle.
\end{equation}
             The support function of Minkowski of the relative normalization $\vect{y}$ is defined by
\begin{equation}\label{1.35}
            q \coloneqq \langle \vect{\xi},\vect{y}\rangle \colon M \rightarrow \mathbb{R} ,\quad q\in C^{r-1}(M),
\end{equation}
            and, by virtue of \eqref{1.15}, never vanishes on $M$. 
            
            Conversely, the relative normalization is determined by means of the support function through
\begin{equation}\label{1.40}
            \vect{y}=\nabla^{III}\!\!\left( q,\, \vect{x} \right) + q \,\vect{\xi},
\end{equation}
            where $\nabla^{III}$ denotes the first Beltrami-operator with respect to the third fundamental form $III$ of $\varPhi$ (see~ \cite[p.~197]{fM89}, \cite{Stamatakis6}).

            Because of \eqref{1.20} the following relations are valid
\begin{equation}\label{1.45}
            \vect{X} = q^{-1} \vect{\xi}, \quad  G_{ij}=q^{-1}h_{ij},  \quad G^{(ij)} = q\, h^{(ij)},
\end{equation}
            where $h^{(ij)}$ and $G^{(ij)}$ are the inverse of the tensors $h_{ij}$ and  $G_{ij}$, respectively. 
            
            From now on we shall use $G_{ij}$ for ``raising and lowering" the indices in the sense of the classical tensor notation.

            Let $\nabla^{G}_{i}$ denote the covariant derivative corresponding to $G$ in the direction $u^i$.
            
            By
\begin{equation}                \label{1.50}
            A_{ijk}\coloneqq \langle \vect{X},\,\nabla^{G}_{k} \, \nabla^{G}_{j}\,\partial_i \vect{x} \rangle
\end{equation}
            the (symmetric) Darboux-tensor is defined.
            
            It gives occasion to define the Tchebychev-vector
\begin{equation}                \label{1.55}
            \vect{T} \coloneqq T^{m}\, \partial_m \vect{x},\quad \text{where\quad }T^{m} \coloneqq \frac{1}{n}A_{i}^{im},
\end{equation}%
            of the relative normalization $\vect{y}$ and the Pick-invariant
\begin{equation}                \label{1.60}
            J \coloneqq \frac{1}{n \left( n-1 \right)} \, A_{jkl} \, A^{jkl}.
\end{equation}
            We consider the bilinear form
\begin{equation}\label{1.65}
            B \coloneqq \langle \DIFF \vect{y}, \DIFF \vect{X}  \rangle.
\end{equation}
            For its coefficients $B_{ij}$ we have
\begin{equation}\label{1.70}
            B_{ij} = \langle  \partial_i \vect{y}, \partial_j \vect{X}\rangle =  - \langle  \partial_j \partial_i \vect{y}, \vect{X}\rangle.
\end{equation}
            Then the following Weingarten type equations  are valid
\begin{equation}\label{1.80}
            \partial_i \vect{y} = -B_{i}^{j} \,  \partial_j  \vect{x}.
\end{equation}
            Let $T_P \varPhi$ be the tangent vector space of $\varPhi$ at the  point $P \in \varPhi$. By means of \eqref{1.80} the relative shape (or Weingarten) operator
\begin{equation*}
            \omega \colon T_P \varPhi \rightarrow T_P \varPhi
\end{equation*}
            of the relatively normalized hypersurface \rn{} at $P$ is defined such that
\begin{equation}\label{1.90}
            B(\vect{u},\vect{v}) = G(\omega( \vect{u}),\vect{v})
\end{equation}
            for tangent vectors $\vect{u},\vect{v}$ at $P$ (see \cite[p.~66]{Simon1991}), or equivalently such that
\begin{equation}\label{1.92}
            \omega(\partial_i \vect{x}) = - \partial_i \vect{y}.
\end{equation}
            Special mention should be made of the fact that the relative differential geometry includes both the Euclidean one, which arises for $q = 1$, or equivalently for $\vect{y} = \vect{\xi}$, and the equiaffine one, which is based upon the equiaffine normalization $\vect{y}_{\textsc{aff}}$.
            The last normalization is defined, on account of \eqref{1.45}, by means of the equiaffine support function
\begin{equation}\label{1.95}
            q_{\textsc{aff}} \coloneqq |\widetilde{K}|^\frac {1} {n+2}.
\end{equation}
            We consider the Tchebychev-function  \cite{Stamatakis4}
\begin{equation}                \label{1.105}
            \varphi \coloneqq \left( \frac {q}{q_\textsc{aff}} \right)^\frac{n+2}{2n},
\end{equation}
            of the relative normalization $\vect{y}$.
            It is well known, that for the components of the Tchebychev-vector the relation \cite[p.~199]{fM89}
\begin{equation}                \label{1.115}
            T^{i}= G^{(ij)} \,\partial_j( \ln \varphi)
\end{equation}
            holds. Hence, by (\ref{1.45}), we obtain
\begin{equation}                \label{1.120}
            \vect{T}= q \,  \nabla^{II}\left( \ln \varphi,\vect{x} \right).
\end{equation}
            We note that the Tchebychev-vector vanishes identically \Iff{} the Tchebychev-function $ \varphi $ is constant, or, by \eqref{1.105}, \Iff{} $q = \lambda \, q_{\textsc{aff}}$, $ \lambda \in \mathbb{R}-\{0\}$, which means that the relative normalization $\vect{y}$ is homothetic (i.e. constantly proportional) to the equiaffine normalization $\vect{y}_{\textsc{aff}}$.

            The Laplace-normal vector of $\vect{y}$ is introduced by the $C^{1}$-mapping $\vect{L} \colon M \rightarrow \mathbb{R}^{n+1}$ given by
\begin{equation}                \label{1.125}
            \vect{L} = \frac{\triangle^{G} \vect{x}}{n},
\end{equation}
            where $\triangle^{G}$ is the second Beltrami-operator with respect to the relative metric $G$.
            
            It is well known, that it fulfils the relation
\begin{equation}                \label{1.130}
            \vect{L} = \vect{T} + \vect{y}.
\end{equation}
            Thus, by \eqref{1.40} and \eqref{1.120}, we have
\begin{equation}                \label{1.135}
            \vect{L} = q \left[ \nabla^{II}\!\!\left(\ln \frac{\varphi}{q}, \vect{x} \right) + \vect{\xi}\right].
\end{equation}

            The real eigenvalues $k_i$ of the relative shape operator $\omega$ of \rn{} are called the relative principal curvatures of \rn{} and their inverses are called the relative radii of curvature and denoted by $R_i$. The associated eigenvectors of the real eigenvalues are called the relative principal  vectors and the corresponding directions are called the relative principal directions of \rn.

            In what follows we shall consider only relatively normalized hypersurfaces such that their relative shape operator has $n$ real eigenvalues $k_1, \dotsc, k_n$ (not necessarily being all different).
            Their averaged
            elementary symmetric functions
\begin{equation}                    \label{1.137}
            H_r \coloneqq \binom{n}{r}^{-1} \!\!\! \sum_{1 \leq i_1 \dotsb  < i_{r} \leq n} \! \! \!\!k_{i_1}  \dotsm k_{i_{r}}, 1 \leq r \leq n,
\end{equation}
            are called the relative mean curvature functions of \rn.
           
            Especially the first relative mean curvature
\begin{equation}                \label{1.138}
            H \coloneqq H_1 = \frac{1}{n} \left(k_1 +  \dotsb + k_n \right) = \frac{1}{n} \TR \left( B_i^j \right)
\end{equation}
            is called the relative mean curvature and the $n$-th  relative mean curvature
\begin{equation}                \label{1.139}
            K \coloneqq H_n =  k_1 \dotsm k_n = \det\left( B_i^j \right)
            \end{equation}
            is called the relative  curvature of \rn{}.

            Let $\delta_i^j$ be the Kronecker delta. The relative principal curvatures are the roots of the characteristic polynomial
\begin{equation}                \label{1.140}
            P_\omega(k) = \det \left( B_i^j - k \,\delta_i^j \right),
\end{equation}
            of $\omega$, or, what is the same, the roots of the equation
\begin{equation}\label{1.143}
            \det \left( B_{ij} - k \,G_{ij} \right) = 0.
\end{equation}
            $P_\omega(k)$ can be written in terms of the relative mean curvatures as ($H_0 \coloneqq 1$)
\begin{equation}                \label{1.145}
            P_\omega(k) = \sum_{r=0}^n\binom{n}{r}\,H_r \, (-k)^{n-r},
            \end{equation}
where $H_0 \coloneqq 1$.

\section{Relatively parallel hypersurfaces}\label{Section 2}
Let a relatively normalized hypersurface \rn{} in the space $\mathbb{R}^{n+1}$ be given. 

In what follows we suppose that the relative normalization $\vect{y} = \vect{y}(u^i)$ is an immersion. Then the so called relative image $\overline{\varPhi} \coloneqq (M, \vect{y})$ of \rn{} is a hypersurface in $\mathbb{R}^{n+1}$ which besides possesses the same unit normal vector field, $\vect{\xi}(u^i)$, as $\varPhi$.

We consider the one-parameter family of mappings $\vect{x}_\mu \colon M \rightarrow \mathbb{R}^{n+1}$ which are defined by
\begin{equation}\label{2.10}
            \vect{x}_\mu(u^i) = \vect{x}(u^i) + \mu \, \vect{y}(u^i),
            \end{equation}
            where $\mu$ is a real nonvanishing constant.
            From \eqref{1.80} and \eqref{2.10} we obtain
\begin{equation}\label{2.15}
            \partial_i \vect{x}_\mu  = \left(\delta^j_i - \mu \, B^j_i\right)\, \partial_j \vect{x}.
\end{equation}
            Then it is readily verified that the vector product of the partial derivatives $\partial_i \vect{x}_\mu $ satisfies the relation
\begin{equation}\label{2.20}
            \partial_1 \vect{x}_\mu \times \dotsb \times \partial_n  \vect{x}_\mu = A(\mu) \;
            \Big( \partial_1 \vect{x}   \times \dotsb \times \partial_n  \vect{x}\Big),
\end{equation}
            where
\begin{equation}\label{2.25}
            A(\mu)\coloneqq  \det \!\left(\delta_i^j -\mu \, B_i^j \right).
\end{equation}
            In what follows we suppose that $A(\mu) \neq 0$ everywhere on $M$.
            Then the one-parameter family \eqref{2.10} consists of $C^r$-immersions.
            In this way we obtain the one-parameter family
\begin{equation*}
            \mathcal{F} \eqqcolon \left \{ \, \varPhi_\mu \eqqcolon (M,\vect{x}_\mu) \mid \vect{x}_\mu = \vect{x} + \mu \, \vect{y},\, \mu \in \mathbb{R} \setminus\{0\} \,  \right  \}
\end{equation*}
            of $C^r$-hypersurfaces.
            To the point $P(u^i_0)$ of $\varPhi$ corresponds the point $P_\mu(u^i_0)$ of $\varPhi_\mu$ so that their position vectors are $\vect{x}(u^i_0)$ and  $\vect{x}_\mu(u^i_0)$, respectively.

            From \eqref{2.20} we infer that the tangent hyperplanes to each member of the family $\mathcal{F}$ and to $\varPhi$ at corresponding points are parallel; $\varPhi$ and every $\varPhi_\mu \in \mathcal{F}$ are in Peterson correspondence \cite{Chakmazyan}.

            Furthermore, as we can see immediately by using \eqref{2.10}, the relations \eqref{1.15} are valid as well if the parametrization $\vect{x}(u^i)$ of $\varPhi$ is replaced by the parametrization $\vect{x}_\mu(u^i)$ of $\varPhi_\mu \in \mathcal{F}$.
            Therefore $\vect{y}$ is a relative normalization for each member of $\mathcal{F}$.
            
            We call each relatively normalized hypersurface $\left( \varPhi_\mu ,\vect{y} \right)$ 
            a relatively parallel hypersurface to \rn.
            Throughout what follows, we shall freely use for $\mu$  the expression ``relative distance".

            On account of \eqref{1.140} and \eqref{1.145} we find
                \begin{equation}                    \label{2.30}
                A  = \sum_{r = 0}^n\binom{n}{r}\,H_r \, (-\mu)^{r},
                \end{equation}
which, because of \eqref{1.137}, can be expressed by means of the relative principal radii of curvature $R_1, \dotsc, R_n$ of \rn{} as follows
                \begin{equation}                    \label{2.35}
                A = (-1)^n \,K \left(\mu - R_1 \right) \dotsm (\mu - R_n).
                \end{equation}
            By means of \eqref{1.35}, (\ref{1.45}a) and  \eqref{1.70} it is clear that \rn{} \emph{and every relatively parallel hypersurface $\left( \varPhi_\mu ,\vect{y} \right)$ to \rn{} have in common}
\begin{enumerate}[(a)]
  \item   \emph{the relative image} $\overline{\varPhi}$,
  \item   \emph{the support function} $q$,
  \item  \emph{the covector $\vect{X}$ of their tangent vector spaces} and
  \item  \emph{the covariant coefficients of their relative shape operators}, i.e.
                \begin{equation}                    \label{2.40}
                B_{ij} = B^*_{\phantom{^*}ij}.
                \end{equation}
\end{enumerate}
                \begin{remark}
If the relative normalization $\vect{y}$ of $\varPhi$ is the Euclidean one \TU{(}$\vect{y} = \vect{\xi}$\TU{)}, then $q = 1$ and vice versa \TU{(}cf.~ \eqref{1.15} and \eqref{1.25}\TU{)}. In this case the concept of the relatively parallel hypersurfaces reduces to the Euclidean one.
                \end{remark}
\section{Formulae apparatus}\label{Section 3}
 In this section we investigate relationships between the second fundamental forms and the relative metrics of a given relatively normalized hypersurface and a relatively parallel to it.
 
 For simplicity we denote a relatively parallel hypersurface $\left( \varPhi_\mu ,\vect{y} \right)$ at a relative distance $\mu$ by \rns.
            Analogously, we mark all the corresponding quantities induced by \rns{} with  an asterisk and we refer by (\#*) to the formula, which, on this modification, flows from formula (\#).
            
            Let
                        \begin{equation*}
                        \overline{I} = \overline{g}_{ij}\DIFF u^i \DIFF u^j, \quad \overline{II} = \overline{h}_{ij}\DIFF u^i \DIFF u^j,
                        \end{equation*}
be the first and the  second fundamental forms of the relative image $\overline{\varPhi}$ of \rn.

Because of \eqref{1.80} we obtain
                        \begin{equation}                \label{3.01}
                        \overline{g}_{ij} = B_i^k \,B_j^m \, g_{km},
                        \end{equation}
where $g_{ij}$ denote the coefficients of the first fundamental form of $\varPhi$. 

Recalling the Euclidean Weingarten equations
                        \begin{equation}   \label{3.02}
                        \partial_j \vect{\xi} = -h_{jk} \, g^{(km)} \, \partial_m \vect{x},
                        \end{equation}
and substituting it as well as $ \partial_i\vect{y}$ from \eqref{1.80} in $\overline{h}_{ij} = -\langle \partial_i\vect{y}, \partial_j \vect{\xi}\rangle$ we get
                        \begin{equation}            \label{3.03}
                        \overline{h}_{ij} = -B_i^k \, h_{kj}.
                        \end{equation}
On account of (\ref{1.45}b) and \eqref{3.03} we arrive at
                        \begin{equation}        \label{3.05}
                        \overline{II} = -q \,B.
                        \end{equation}
By combining \eqref{3.01} and \eqref{3.03} we can derive the following relation between the Gaussian curvatures $\widetilde{K}$ and $\widetilde{\overline{K}}$ of $\varPhi$ and $\overline{\varPhi}$ respectively and the relative curvature $K$ of \rn{}:
                        \begin{equation}\label{3.10}
                        K = (-1)^n \, \frac{\widetilde{K}}{\widetilde{\overline{K}}}.
                        \end{equation}
In analogy to the computation above we get by using \eqref{2.10} and \eqref{3.05} the second fundamental form of $\varPhi^*$:
                \begin{equation}            \label{3.15}
                II^* = II - \mu \,q\, B.
                \end{equation}
For the relative metric $G^*$ of \rns{} we find by combining \eqref{1.25}, (\ref{1.25}*), \eqref{1.65} and \eqref{2.10}
                \begin{equation}            \label{3.17}
                G^* = G - \mu \, B.
                 \end{equation}
From \eqref{1.80}, (\ref{1.80}*), \eqref{2.15} and
                \begin{equation}            \label{3.18}
                \partial_i \vect{x}  = \left(\delta^j_i + \mu \, B^{*j}_{\phantom{^*}i} \right) \, \partial_j \vect{x}^* \quad  \forall \,\, i = 1, \dotsc,n,
                \end{equation}
we obtain
                \begin{equation}            \label{3.20}
                B^j_i = B_{\phantom{^*}i}^{*k}\,(\delta^j_k - \mu \, B^j_k),
                 \qquad B^{*j}_{\phantom{^*}i} = B_i^{k} \, (\delta^j_k + \mu \, B^{*j}_{\phantom{^*}k}) \quad \forall \,\, i,j = 1,\dotsc,n
                \end{equation}
and
                \begin{equation}            \label{3.22}
                B^k_i\, B_{\phantom{^*}k}^{*j} = B^{*k}_{\phantom{^*}i} \, B_k^j  \quad \forall \,\, i,j = 1,\dotsc,n.
                \end{equation}
By combining \eqref{1.139},  (\ref{1.139}*), \eqref{2.10} and (\ref{3.20}a)  we have
                \begin{corollary}
                The relative curvature of \rns{} is given by
                \begin{equation}                        \label{3.23}
                K^* = \frac{K}{A}.
                \end{equation}
                \end{corollary}
\section{The shape operator of a relatively parallel hypersurface}\label{Section 4}
In this section we study the relative shape operator and the relative mean curvature functions of a relatively parallel hypersurface to \rn{}.
 
                Let $\omega$ be the relative shape operator of \rn{} at a point $P \in \varPhi$ and $\omega^*$ that of the relatively parallel hypersurface \rns {}  at relative distance $\mu$  at the corresponding point $P^* \in \varPhi^*$.
                From  \eqref{1.80}, (\ref{1.80}*), \eqref{1.92} and (\ref{1.92}*) we take
\begin{equation}                \label{3.231}
                \omega^*(\partial_i \vect{x}^*) = \omega(\partial_i \vect{x}) \quad \forall \,\, i = 1,\dotsc,n,
\end{equation}
                in other words the mapping
\begin{equation}                    \label{3.24}
                \mu \mapsto \omega^*
\end{equation}
does not depend on $\mu$. This fact can be written as
\begin{equation}                        \label{3.241}
                \frac{\DIFF \, \omega^*(\partial_i \vect{x}^*)}{\DIFF \, t} = 0.
\end{equation}
                We consider in addition \eqref{2.15} and \eqref{3.18}--\eqref{3.22} and we have
\begin{equation*}
\begin{split}
                \omega^*(\partial_i \vect{x}) &= (\delta_i^j + \mu \, B^{*j}_{\phantom{^*}i}) \, \omega^*(\partial_j \vect{x}^*) = (\delta_i^j + \mu \, B^{*j}_{\phantom{^*}i}) \, B_j^m \, \partial_m \vect{x}
                = (B_i^m + \mu \, B^{*j}_{\phantom{^*}i}B_j^m) \, \partial_m \vect{x} \\
                &=
                (B_i^m + \mu \,B_i^j B^{*m}_{\phantom{^*}j}) \, \partial_m \vect{x} = B_i^j \,(\delta_j^m + \mu \, B^{*m}_{\phantom{^*}j}) \, \partial_m \vect{x} = B^{*m}_{\phantom{^*}i} \, \partial_m \vect{x},
\end{split}
\end{equation*}
                that is
\begin{equation}                    \label{3.25}
                \omega^*(\partial_i \vect{x}) = B^{*m}_{\phantom{^*}i} \, \partial_m \vect{x} \quad \forall \,\, i = 1,\dotsc,n.
\end{equation}
In an analogous manner we obtain
\begin{equation}                    \label{3.26}
                \omega(\partial_i \vect{x}^*) = B^{m}_{i} \, \partial_m \vect{x}^* \quad \forall \,\, i = 1,\dotsc,n.
\end{equation}
We can now prove the following propositions
\begin{proposition}
                For the relative shape operators $\omega$ and $\omega^*$ of \rn{} and \rns{}, respectively, there hold the following properties: 
\begin{enumerate}[\TU{(a)}]
  \item They commute.
  \item They are related through
                \begin{equation}                            \label{3.27}
                \omega^* = (\ID - \mu \, \omega )^{-1} \, \omega, 
                \end{equation}
                \end{enumerate}
  where $\ID$ is the identity transformation.
\end{proposition}
\begin{proof}
                The first property is a consequence of \eqref{3.22}, \eqref{3.25} and \eqref{3.26}.
                
                For the second we have from \eqref{3.20} and \eqref{3.25}
\begin{equation*}
                (\ID - \mu \, \omega) \, \omega^*(\partial_i \vect{x}) = B^{*m}_{\phantom{^*}i} \, \Big[\partial_m \vect{x} - \mu \, \omega( \partial_m \vect{x})\Big] = B^{*m}_{\phantom{^*}i} \, (\delta_m^j - \mu B_m^j)\, \partial_j \vect{x} = B_i^j \, \partial_j \vect{x} = \omega(\partial_i \vect{x}),
\end{equation*}
                and the proof is completed.
\end{proof}
                From \eqref{1.92} and \eqref{2.10} we have
\begin{equation}                        \label{3.28}
                \frac{\DIFF \, \partial_i \vect{x}^*}{\DIFF \, t} = -\omega(\partial_i \vect{x}).
\end{equation}
                Taking into account \eqref{3.231}, \eqref{3.241} and \eqref{3.28} we then find
\begin{equation}                    \label{3.281}
                \frac{\DIFF \, \omega^*}{\DIFF \, t}(\partial_i \vect{x}^*) = -\omega^* \Big[ \frac{\DIFF \, \partial_i \vect{x}^*}{\DIFF \, t} \Big] = \omega^* \Big[\omega(\partial_i \vect{x}) \Big] =  \omega^* \Big[\omega^* (\partial_i \vect{x}^*) \Big] = (\omega^{*})^2(\partial_i \vect{x}^*).
\end{equation}
                Hence
\begin{proposition}
                The mapping \eqref{3.24} fulfills the differential equation
\begin{equation}                \label{3.29}
                \frac{\DIFF \, \omega^*}{\DIFF \, t} = \omega^2.
\end{equation}
\end{proposition}
\begin{proposition}
                Let $k_1,\dotsc,k_n$ be the relative principal curvatures at a point $P \in \varPhi$. Then the relative principal curvatures of the relatively parallel hypersurface \rns {}  at relative distance $\mu$  at the corresponding point $P^* \in \varPhi^*$ are the following:
\begin{equation}                \label{3.30}
                k^*_i = \frac{k_i}{1-\mu \, k_i}, \quad i = 1, \dotsc, n.
\end{equation}
\end{proposition}
\begin{proof}
                The  relative principal curvatures of \rns{} are the solutions of (\ref{1.143}*), that is of
\begin{equation*}
                \det \left( B^*_{\phantom{^*}ij} - k^* \,G^*_{\phantom{^*}ij} \right) = 0,
\end{equation*}
                which on account of \eqref{2.40} and \eqref{3.17} can be written as
\begin{equation}                    \label{3.35}
                 \det \left( B_{ij} - \frac{k^*}{1 + \mu \, k^*} \, G_{ij}\right) = 0.
\end{equation}
                But the solutions of the latter are the $k_i, i = 1, \dotsc,n$, cf.~\eqref{1.143}.
                Consequently, the relative principal curvatures $k_i$ and $k_i^*$, of \rn{} and \rns{}, respectively, are connected with the relation
\begin{equation}                    \label{3.40}
            k_i =  \frac{k_i^*}{1 + \mu \, k_i^*},
\end{equation}
            which is equivalent to \eqref{3.30}.
\end{proof}
\begin{proposition}
                \rns{} has constant relative mean curvature for each $\mu$ \Iff{} \rn{} has constant relative principal curvatures.
\end{proposition}
            The relative radii of curvature $R_i^*$ of \rns{} are obtained immediately from \eqref{3.30}:
\begin{equation}                        \label{3.35}
            R_i^* = R_i - \mu, \quad i = 1, \dotsc, n.
\end{equation}
            Thus
\begin{equation}                        \label{3.40}
             \sum _{1 \leq i \leq n} \! R_i = n \, \mu + \sum _{1 \leq i \leq n} \! R_i^*.
\end{equation}
            Let $\vect{u}_1,\dotsc \vect{u}_n$ be relative principal vectors which belong to the relative principal curvatures $k_1,\dotsc,k_n$, respectively, at a point $P \in \varPhi$. Then
\begin{equation}                        \label{3.45}
            \omega(\vect{u}_i) = k_i \,\vect{u}_i \quad \forall \,\,  i = 1,\dotsc, n.
\end{equation}
            It is well known that different relative principal vectors are relatively orthogonal with respect to $B$, that is for $\vect{u}_r = u_r^i \, \partial_i \vect{x}$ the following relation holds
\begin{equation}                        \label{3.50}
            B(\vect{u}_r,\vect{u}_s) = B_{ij} \, u_r^i \, u_s^j = 0 \quad \forall \,\,  r,s = 1,\dotsc,n
\end{equation}
            (see~\cite[p.~10]{Manhart1}). 
            
            We consider the tangent vectors
\begin{equation}                        \label{3.55}
                \vect{u}^*_i = k_i^* \, \vect{u}_i, \quad i = 1, \dotsc, n,
\end{equation}
                which are relatively orthogonal with respect to $B^*$, cf.~\eqref{2.40}. 
                
                From \eqref{3.25}, \eqref{3.30} and \eqref{3.45} we find
\begin{equation}                        \label{3.60}
                \omega^*(\vect{u}^*_i) =  k^*_i \, \vect{u}^*_i, \quad i = 1, \dotsc, n.
\end{equation}
                Consequently each vector $\vect{u}^*_i = k_i^* \, \vect{u}_i$, $i = 1, \dotsc, n$, is a relative principal vector of \rns{} which belongs to the relative principal curvature $k_i^*$. Thus we have proved the following
\begin{proposition}
                Let $\vect{u}_1,\dotsc \vect{u}_n$ be relative principal vectors at a point $P \in \varPhi$. Then the vectors $k^*_1 \, \vect{u}_1, \dotsc, k^*_n \, \vect{u}_n$ are relative principal vectors at the corresponding point $P^* \in \varPhi^*$ of the relatively parallel hypersurface \rns {} of \rn{}  at relative distance $\mu$.
\end{proposition}
                We write the function $A$ (see \eqref{2.20}) as
\begin{equation}                    \label{3.80}
                A = (1 - \mu \, k_1) \dotsm (1 - \mu \, k_n).
\end{equation}
                Viewing it as a function of $\mu$ we observe that
\begin{equation}                    \label{3.140}
                \frac{\DIFF^{\,\,s}\!\! A}{\DIFF  \,\mu^{s}} =  (-1)^s \,\,\, s\,! \sum_{1 \leq i_1 < \dotsb  < i_{s} \leq n} \! \! \!\!k_{i_1} \dotsm k_{i_{s}} \!\!\ \prod_{\substack{1 \leq j \leq n \\ j \neq i_1,\dotsc ,i_{s}}}\!\!\!(1 - \mu \, k_{j}).
\end{equation}
                By using \eqref{3.80} one gets
\begin{equation}                    \label{3.145}
                \frac{(-1)^{s}}{s\,! \,\,A} \cdot \frac{\DIFF^{\,\,s}\!\! A}{\DIFF  \,\mu^{s}} =  \sum_{1 \leq i_1 < \dotsb  < i_{s} \leq n} \!\!\!\frac{k_{i_1}}{1 - \mu \, k_{i_1}} \dotsm \frac{k_{i_{s}}}{1 - \mu \,k_{i_{s}} }.
\end{equation}
%
                By using (\ref{1.137}*) and \eqref{3.30} we may obtain the relative mean curvatures of \rns{} which are the following
                        \begin{equation}                    \label{3.150}
                        H^*_s = \frac{(-1)^{s}}{s\,! \,\, \binom{n}{s} \,\,A} \cdot \frac{\DIFF^{\,\,s}\!\! A}{\DIFF  \,\mu^{s}}, \quad s = 0, 1, \dotsc, n.
                        \end{equation}
                On account of \eqref{2.15} we have
\begin{equation}                    \label{3.155}
            \frac{\DIFF^{\,\,s}\!\! A}{\DIFF  \,\mu^{s}} =  \sum_{s \leq r \leq n} (-1)^r \, \frac{r\,!}{(r - s)\,!} \,\, \binom{n}{r} \, H_r \,\,\, \mu^{r - s}.
\end{equation}
            Inserting the latter in \eqref{3.150} we obtain
\begin{equation}                    \label{3.160}
            H^*_s = \frac{(-1)^{s}}{s\,! \,\, \binom{n}{s} \,\,A} \cdot \sum_{s \leq r \leq n} (-1)^r \, \frac{r\,!}{(r - s)\,!} \,\, \binom{n}{r} \, H_r \,\,\, \mu^{r - s}.
\end{equation}
            For $s = n$ we regain \eqref{3.23}. For $s = n -1$ we find $A \,H_{n-1}^* = H_{n-1} - \mu \, K$, or, on account of \eqref{3.23},
\begin{equation}                            \label{3.162}
            \frac{H_{n-1}^*}{K^*} = \frac{H_{n-1}}{K} - \mu,
\end{equation}
            which essentially is \eqref{3.40}. Moreover we have the following
\begin{proposition}
            Let \rn{} be a relatively normalized hypersurface  of constant sum 
            of its relative principal radii of curvature in the $n$-dimensional Euclidean space $\mathbb{E}^n$. Then every relatively parallel surface to \rn{} has also constant sum of its relative radii of curvature and there is exactly one relatively parallel hypersurface to \rn{} whose ($n - 1$)-th  relative mean curvature vanishes.
\end{proposition}
\begin{proof}
            The first part of the proposition follows from \eqref{3.40}. 
            
            Suppose that $R_1 + \dotsb + R_n = n\,c = const.$ and consider the relatively parallel hypersurface at relative distance  $\mu = c$. But
\begin{equation}                        \label{3.163}
             \sum _{1 \leq i \leq n} \! R_i = \frac{n\,H_{n - 1}}{K}.
\end{equation}
            Then from \eqref{3.162} we obtain $H^*_{n - 1} = 0$.
\end{proof}
            For $s = 1$ we get the relative mean curvature $H^*$ of \rns{}:
\begin{equation}                    \label{3.165}
            H^* = \frac{1}{n\,A} \cdot \sum_{1 \leq r \leq n} (-1)^{r + 1} \,\,r \,\binom{n}{r}\,H_r \, \mu^{r-1},
\end{equation}
            or
\begin{equation}                    \label{3.170}
\begin{split}
            H^* = \frac{1}{n\,A} \cdot & \Bigg[ n \, H - 2 \binom{n}{2} \, \mu \, H_2 + 3 \binom{n}{3}\,\mu^2 \,H_3 + \dotsb \\
            &{}+ 
            (-1)^n(n-1)\binom{n}{n-1}\,\mu^{n-2} \, H_{n-1} + (-1)^{n + 1} \,n\, \mu^{n-1} \,K \Bigg].
\end{split}
\end{equation}
            The relatively normalized surfaces \rn{} and \rns{} have common relative image. Therefore, besides \eqref{3.10}, we have
\begin{equation}                    \label{3.175}
            K^* = (-1)^n \, \frac{\widetilde{K}^*}{\widetilde{\overline{K}}},
\end{equation}
                hence by means of \eqref{3.10}
\begin{equation}                    \label{3.180}
                        \frac{\widetilde{K}^*}{K^*} = \frac{\widetilde{K}}{K},
\end{equation}
                and we arrive at the following
\begin{proposition}
                The function $\widetilde{K}/K$ remains invariant by the transition to anyone of the relatively parallel hypersurfaces of \rn.
\end{proposition}
                By combining  \eqref{3.23} and \eqref{3.180} we find
\begin{equation}\label{3.200}
                \frac{\widetilde{K}^*}{\widetilde{K}} = \frac{1}{A}.
\end{equation}
\section{The affine normalization of a relatively parallel hypersurface}\label{Section 5}
In this section we study the affine normalization, the Tchebychev-vector and the Laplace-vector of a relatively parallel hypersurface.
               
                Let $q^*_{\textsc{aff}}$ be the support function of the affine normalization $\vect{y}^*_{\textsc{aff}}$ of $\varPhi^*$.
                On account of \eqref{1.95}, (\ref{1.95}*) and \eqref{3.200} we find
\begin{equation}\label{3.205}
                q^*_{\textsc{aff}} = \left| A \right|^{\frac{-1}{n+2}} \, q_{\textsc{aff}}.
\end{equation}
                From \eqref{1.105}, (\ref{1.105}*) and \eqref{3.205} results the Tchebychev-function
\begin{equation}                \label{3.210}
                \varphi^* = \left| A \right|^{\frac{1}{2n}} \, \varphi.
\end{equation}
                of $\varPhi^*$. 
                
Let 
\begin{equation*}
  III \coloneqq e_{ij}\DIFF u^i \DIFF u^j
\end{equation*}
 be the third fundamental form of $\varPhi$ and  $e^{(ij)}$ the inverse of the tensors $e_{ij}$. On account of the well known relations
\begin{equation*}
                e^{(ij)} = h^{(ir)} \, h^{(js)} \, g_{rs}
                \end{equation*}
and \eqref{3.02} one can immediately verify the following relation
\begin{equation}                            \label{3.215}
                \nabla^{II}\!\!\left( f,\, \vect{x} \right) = - \nabla^{III}\!\!\left( f,\, \vect{\xi} \right)
\end{equation}
                for a $C^1$-function $f(u^i)$.
                But $\varPhi$ and $\varPhi^*$ have common Euclidean normal vector field. So, by means of \eqref{1.120}, (\ref{1.120}*), \eqref{1.130}, (\ref{1.130}*) and \eqref{3.215}, we obtain the following relation between the Tchebychev-vectors and the Laplace normal vectors, respectively:
\begin{align}
                \vect{T}^* =   \vect{T} - \frac{q}{2n} \, \nabla^{III}\!\!\left( \ln |A|,\, \vect{\xi} \right),     \label{3.220} \\
                \vect{L}^* =   \vect{L} - \frac{q}{2n} \, \nabla^{III}\!\!\left( \ln |A|,\, \vect{\xi} \right) \label{3.222}.
\end{align}
                In view of \eqref{1.40}, (\ref{1.40}*), \eqref{3.215} and recalling that $\varPhi$ and $\varPhi^*$ have common spherical image, we find the affine normalization  $\vect{y}^*{\textsc{aff}}$ of $\varPhi^*$:
\begin{equation}            \label{3.225}
                \vect{y}^*{\textsc{aff}}   = |A|^{\frac{-1}{n+2}} \,\, \vect{y}_{\textsc{aff}} +  q_{\textsc{aff}} \,\, \nabla^{III}\!\!\left( |A|^{\frac{-1}{n+2}},\, \vect{\xi} \right).
\end{equation}
\section{Parallelism of the affine normals}\label{Section 5}

                We consider a relatively parallel hypersurface \rns{} at relative distance $\mu$ to the given relatively normalized hypersurface \rn{} and we suppose that the  affine normals of the surfaces $\varPhi$ and $\varPhi^*$ at corresponding points are parallel. This occurs \Iff{}
\begin{equation}                \label{4.05}
                    \vect{y}^*{\textsc{aff}}   = c \,\vect{y}_{\textsc{aff}}, \quad    c \in \mathbb{R},
\end{equation}
                or, on account of \eqref{1.40}, (\ref{1.40}*) and \eqref{3.205}, \Iff{} $A = c_1 \in \mathbb{R}$. Because of \eqref{2.15} the latter is equivalent to
\begin{equation}                    \label{4.10}
                n\, H - \binom{n}{2} \,\mu \, H_2 + 
                \dotsb + (-1)^n \, \binom{n}{n - 1} \,\mu^{n - 2} \, H_{n - 1} + (-1)^{n + 1} \, \mu^{n - 1}\ K  =  = \frac{1 - c_1}{\mu}const.
\end{equation}
                Taking into account the relations \eqref{3.200}, \eqref{3.220} and \eqref{3.222} we arrive at the following proposition
\begin{proposition}
                Let \rn{} be a relatively normalized hypersurface of $\mathbb{R}^{n+1}$ and \rns{} a relatively parallel hypersurface of it. Then the following properties are equivalent:
\begin{enumerate}[\TU{(a)}]
  \item The affine normals of $\varPhi$ and $\varPhi^*$ are parallel.
  \item The Gaussian curvatures of  $\varPhi$ and $\varPhi^*$ are proportional.
  \item The Tchebychev and the Laplace vectors of the normalization $\vect{y}$ with respect to $\varPhi$ and $\varPhi^*$ coincide.
  \item The relative mean curvature functions of \rn{} are connected with a linear relation with constant coefficients of the form \eqref{4.10}.
\end{enumerate}
\end{proposition}

\begin{corollary}
                If one of the conditions \TU{(a)--(d)} of the latter proposition is valid for $n$ different relatively normalized hypersurfaces \rns{} of \rn,  then the relative principal curvatures of \rn{} are constant.
\end{corollary}

\end{document}